\documentclass[10pt]{amsart} 

\usepackage[utf8]{inputenc}
\usepackage[T1]{fontenc}	
\usepackage[french,english]{babel}	

\usepackage{lmodern}			

\usepackage{graphicx}	

\usepackage[a4paper,plainpages=false,colorlinks,linkcolor=bleuFonce,citecolor=rougeFonce,urlcolor=vertFonce,breaklinks]{hyperref}

\usepackage{placeins}
\usepackage{float} 

\usepackage{color}
\definecolor{vertFonce}{rgb}{0,0.5,0}
\definecolor{numLignes}{rgb}{0.17,0.57,0.7}	
\definecolor{gris}{rgb}{0.5,0.5,0.5}
\definecolor{grisFonce}{rgb}{0.2,0.2,0.2}
\definecolor{orange}{rgb}{1,0.65,0.31}		
\definecolor{orangeFonce}{rgb}{1,0.4,0}
\definecolor{bleuFonce}{rgb}{0,0,0.4}
\definecolor{rougeFonce}{rgb}{0.3,0,0}
\definecolor{rougeWord}{rgb}{0.5,0,0}
\definecolor{vertClair}{rgb}{0.8,1,0.8}
\definecolor{rougeClair}{rgb}{1,0.5,0.5}

\usepackage{pict2e}
\usepackage{multido}

\usepackage{amsfonts,amssymb,amsthm,amsmath} 
\usepackage{dsfont}				
\usepackage{mathrsfs}
\usepackage{yfonts}
 
\newtheorem{lem}{Lemma}[section]
\newtheorem{theorem}{Theorem}
\newtheorem{cor}{Corollary}[section]
\newtheorem{prop}{Proposition}

\newtheorem{remark}{Remark}[section]

\newenvironment{thm}[1][]{
\begin{theorem}[#1]
	}{%
	\end{theorem}
}

\usepackage{stmaryrd}
\newcommand		{\subsetArrow}	{\mathrel{\ooalign{$\subset$\cr%
\hidewidth\raise-.087ex\hbox{$_\shortrightarrow\mkern-1.5mu$}\cr}}}
\newcommand		{\subsetarrow}	{\mathrel{\ooalign{$\subset$\cr%
\hidewidth\raise-1.45ex\hbox{$\vec{}\mkern6mu$}\cr}}}

\newcommand		{\R}		{\mathbb R}			
\renewcommand	{\SS}		{\mathds S}			

\newcommand		{\lt}			{\left}				%
\newcommand		{\rt}			{\right}			%
\renewcommand	{\(}			{\lt(}
\renewcommand	{\)}			{\rt)}
\newcommand		{\n}[1]			{\lt|{#1}\rt|}

\newcommand		{\Nrm}[2]		{\lt\|{#1}\rt\|_{#2}}

\renewcommand		{\d}		{\mathrm{d}}		
\newcommand			{\dd}		{\,\d}				
\newcommand			{\dpt}		{\partial_t}
\newcommand			{\dpr}		{\partial_r}
\newcommand			{\dpth}		{\partial_\theta}
\newcommand			{\dpz}		{\partial_z}
\newcommand			{\dt}		{\frac{\d}{\d t}}	

\DeclareMathOperator{\divg}		{div}				
\DeclareMathOperator{\curl}		{curl}				

\newcommand		{\Div}[1]		{\divg\!\( #1 \)}	
\newcommand		{\Curl}[1]		{\curl\!\( #1 \)}	

\newcommand		{\init}			{\mathrm{in}}

\newcommand		{\eps}			{\varepsilon}

\newcommand		{\cC}			{\mathcal{C}}
\newcommand     {\cA}			{\mathcal{A}}

\title[Instability for Axisymmetric Blow-up Solutions to Euler]{Instability for Axisymmetric Blow-up Solutions to Incompressible Euler Equations}

\author[Lafleche]{Laurent Lafleche}
\address[Laurent Lafleche]{\newline Department of Mathematics, \newline The University of Texas at Austin, Austin, TX 78712, USA}
\email{lafleche@math.utexas.edu}

\author[Vasseur]{Alexis F. Vasseur}
\address[Alexis F. Vasseur]{\newline Department of Mathematics, \newline The University of Texas at Austin, Austin, TX 78712, USA}
\email{vasseur@math.utexas.edu}

\author[Vishik]{Misha Vishik}
\address[Misha Vishik]{\newline Department of Mathematics, \newline The University of Texas at Austin, Austin, TX 78712, USA}
\email{vishik@math.utexas.edu}

 
 \thanks{\textbf{Acknowledgment.}  
A. Vasseur was partially supported by the NSF grant: DMS 1907981. }

\begin{document}

\bibliographystyle{plain}

\begin{abstract} 
It is still not known whether a solution to the incompressible Euler equation, endowed with  a smooth initial value, can blow-up in finite time. In  [{\em  Comm. Math. Phys.}, 378:557--568, 2020] it has been shown that, if it exists, such a solution becomes linearly unstable close to the blow-up time. In this paper, we show that the same phenomenon holds even in the more rigid axisymmetric case. To obtain this result, we first prove a blow-up criterion involving only the toroidal component of the vorticity. The instability of blow-up profiles is also investigated. 
\end{abstract}

\maketitle

\bigskip

\renewcommand{\contentsname}{\centerline{Table of Contents}}
\setcounter{tocdepth}{2}	
\tableofcontents


\bigskip
\section{Introduction}

\label{sec:intro}

	Let $e_z$ be the vertical unit vector in $\R^3$, and let $\Omega$ be either $\R^3$ or any smooth bounded subset of $\R^3$ invariant by rotations of axis $e_z$. We then consider the following  incompressible Euler equations in this domain:
	\begin{equation}\label{eq:Euler}
		\begin{array}{rll}
			\dpt u + u\cdot\nabla u + \nabla P &=& 0, 
			\\
			\Div{u} &=& 0, 
		\end{array}
		\qquad\qquad 0<t<T^*, \, x\in \Omega, 
	\end{equation}
	for $T^*>0$. If $\Omega$ has a boundary, we supplement the equation with the impermeable boundary condition:
	\begin{equation*}
		u\cdot n=0, \qquad \qquad 0<t<T^*,\, x\in \partial\Omega,
	\end{equation*}
	where $n$ is normal vector at the boundary.

\vskip0.3cm
	
	We consider  solutions to~\eqref{eq:Euler} which are axisymmetric flows. It means that they are defined via three functions depending only, in space, on the toroidal variables $\tilde x=(r,z)$ as
	\begin{equation}\label{eq:sym}
		u(t,x)= u_r(t,\tilde x)\,e_r+ u_\theta(t,\tilde x)\,e_\theta + u_z(t,\tilde x)\,e_z,
	\end{equation}
	where $x=(r\cos\theta, r\sin \theta, z)=r e_r+ze_z$, and 
	\begin{equation*}
		e_r=(\cos\theta,\sin\theta,0), \qquad e_\theta=(-\sin\theta,\cos\theta,0),\qquad e_z=(0,0,1).
	\end{equation*}
	We denote $\tilde\Omega$ the 2D domain such that $x\in \Omega$ whenever $\tilde x\in \tilde\Omega$.  With a slight abuse of notation, we  use the same notation $u$ both for the function of the variable $x\in \Omega$, and for  $(u_r,u_\theta,u_z)$ as function of the toroidal variable $\tilde x\in \tilde\Omega$.  In particular $L^p(\tilde{\Omega})$ will denote a Lebesgue space on a domain of dimension $2$ while $L^p(\Omega)$ will denote a Lebesgue space on a domain of dimension $3$.	

\vskip0.3cm

	The axisymmetric structure~\eqref{eq:sym} is preserved by the system~\eqref{eq:Euler} and, for any axisymmetric $u^0\in H^s(\Omega)$  initial value with $s>5/2$, there exists $T>0$ such that the associated solution $u$ to~\eqref{eq:Euler} verifies
	\begin{equation}\label{hyp:regularity}\tag{H1}
		u\in C^0([0,T), H^s(\Omega))\cap C^1([0,T),H^{s-1}(\Omega)).
	\end{equation}
	Moreover the solution  is axisymmetric and unique on this lifespan (see for instance~\cite{beale_remarks_1984}). Consider $T^*$ the biggest such time $T$.  It is still unknown whether there  exist such a solution which blows up in finite time, that is such that  $T^*$ is  finite. The aim of this paper is to study, in this finite time blow-up scenario, the associated development of instabilities.
	
\vskip0.1cm

	For this, we consider the semigroup of axisymmetric solutions generated by the linearization of the Euler equations~\eqref{eq:Euler} about the solution $u$:
	\begin{equation}\label{eq:linear_Euler}
		\begin{array}{rlll}
			\dpt v + u\cdot\nabla v + v\cdot \nabla u + \nabla Q &=& 0, &\qquad 0<t<T^*,\, x\in \Omega, 
			\\
			\Div{v} &=& 0  &\qquad 0<t<T^*,\, x\in \Omega, 
			\\
			v\cdot n&=&0, &\qquad \text{on } \partial \Omega.
		\end{array}
	\end{equation}
	The solution $v$ is uniquely determined for any initial  initial value in $H^1(\Omega)$ (see Inoue and Miyakawa~\cite{inoue_existence_1979}). 
	We only consider initial values in $H^1_\text{axi}(\Omega)$, the set of axisymmetric functions in $H^1(\Omega,\R^3)$.
	Since $u$ is axisymmetric, the axisymmetric structure is also preserved by this linear equation, and the solution $v$ of Equation~\eqref{eq:linear_Euler} verifies~\eqref{eq:sym}. 

\vskip0.3cm

	In order to apply our result on blow-up profiles, we study  the instability in weighted $L^p$ spaces. Therefore we measure the growth of the semigroup associated to the linearized Euler equation~\eqref{eq:linear_Euler}, in the space of axisymmetric functions, by
	\begin{align*}
		\lambda_{p,\sigma}(t) &:= \sup_{\substack{v(0,\cdot)\in H^1_\text{axi}(\Omega)\\ v \text{ solves~\eqref{eq:linear_Euler}}}} \frac{\Nrm{r^{-\sigma}v(t,\cdot)}{L^p(\Omega)}}{\Nrm{r^{-\sigma}v(0,\cdot)}{L^p(\Omega)}}.
	\end{align*}

	Our main result is the following.
	\begin{thm}[Instability of the Blow-up]\label{thm:bound_omega_lambda}
		Assume $u$ is an axisymmetric solution of~\eqref{eq:Euler} verifying the Hypothesis~\eqref{hyp:regularity} with $s>5/2$, 
		and with initial condition $u^\init$ verifying $r\,u^\init \in L^\infty(\Omega)\cap L^q(\Omega)$ with $q<\frac{6}{5}$. Let $T^*>0$ be the maximal time $T$ such that $u$ verifies Hypothesis~\eqref{hyp:regularity}, and assume that $T^*$ is finite. Then for any $p\in[1,\infty)$ and any $\sigma\in\(-\frac{2}{p'},\frac{2}{p}\)$,
		\begin{align*}
			\lambda_{p,\sigma}(t) \underset{t\to T^*}{\longrightarrow} \infty.
		\end{align*}
	\end{thm}
	
\vskip0.3cm

	This result follows ideas of~\cite{vasseur_blow-up_2019}, where a similar result was proved without weight nor axisymmetric assumptions. 
	The method is based on the WKB expansion method  developed in~\cite{vishik_spectrum_1996, friedlander_instability_1991} to define rigorously the concept of fluid Lyapunov exponent (see also~\cite{friedlander_dynamo_1991} and~\cite{friedlander_nonlinear_1997}).  

\vskip0.3cm 

	This line of work is motivated by the numerical investigations of finite time blow-up solutions of Euler equations~\eqref{eq:Euler}.
	The genuine difficulty to predict finite time blow-ups for compressible models is well documented (see  Hou and Li~\cite{hou_dynamic_2006}, or Kerr~\cite{kerr_bounds_2013} for instance). The result in~\cite{vasseur_blow-up_2019} shows that, if such a finite time blow-up solution exists, it becomes linearly unstable close to the blow-up time.

\vskip0.3cm

	In~\cite{luo_potentially_2014}, Luo and Hou made very precise numerical computations providing strong evidences for the existence of axisymmetric solutions blowing up in finite time. This has been recently backed up mathematically by Hou and Chen~\cite{chen_finite_2019}. The proof follows the theory initiated by Elgindi~\cite{elgindi_finite-time_2019} (see also Elgindi, Ghoul and Masmoudi~\cite{elgindi_stability_2019} for solutions with non vanishing swirl). Note that these solutions constructed mathematically have initial values in $C^{1,\alpha}$ (without  more regularity). 

\vskip0.3cm

	Considering only axisymmetric solutions prevent the non axisymmetric instabilities in the flow. Theorem \ref{thm:bound_omega_lambda} shows that, even in this case, purely axisymmetric instabilities develop at the blow-up time. Therefore it justifies the need of high precision numerical techniques introduced in~\cite{luo_potentially_2014}.

\vskip0.3cm

	Most of  these studies of blow-ups, whether  mathematical or numerical, involve the  control of a local blow-up profile. Consider the rescaling of the solution in the toroidal variables:
	\begin{equation}\label{eq:scaling}
			u(t,\tilde x) = \frac{1}{(T^*-t)^\alpha}\, U\!\(t,\frac{\tilde x-{x}^\circ_t}{(T^*-t)^\beta}\),
	\end{equation}
	where  $(x^{\circ}_t)_{t\in[0,T]}$ is a curve in time with values in $\tilde \Omega$.  The function $U$ is called a blow-up profile if it smooth enough, up to the blow-up time $T^*$. A natural question is then whether the instabilities developing close to the blow-up time induce instabilities on the blow-up profile itself. 
	For this, we consider the rescaled axisymmetric solutions to the linearized Euler equations~\eqref{eq:linear_Euler}:
	\begin{equation}\label{eq:scaling_v}
		v(t,\tilde x) = \frac{1}{(T^*-t)^\alpha}\, V\!\(t,\frac{\tilde{x}-{x}^\circ_t}{(T^*-t)^\beta}\),\qquad \tilde{\Omega}_t= \frac{\tilde\Omega-{x}^\circ_t}{(T^*-t)^\beta}.
	\end{equation}
	And we define the associated quantity:
	\begin{equation}\label{def:scaled_growth}
		{\Lambda}_p(t) := \sup_{v^\init\in H^1_\text{axi}(\Omega), \Nrm{V^\init}{L^p(\tilde{\Omega}_0)}\leq 1} \Nrm{V(t,\cdot)}{L^p(\tilde\Omega_t)}.
	\end{equation}
	Using  the weighted norms considered in Theorem \ref{thm:bound_omega_lambda}, we  can show the following result.
	\begin{cor}\label{cor}
		Let $1\leq p<\infty$, $\alpha\in\R$ and $\beta>0$ be numbers such that
		\begin{equation*}
			\frac{\alpha}{\beta}<1+\frac{4}{p}.
		\end{equation*}
		Assume $u$ is an axisymmetric solution of~\eqref{eq:Euler} verifying~\eqref{hyp:regularity} with $s>5/2$, and with initial condition $u^\init$ verifying $r\,u^\init \in L^\infty(\Omega)\cap L^q(\Omega)$ with $q<\frac{6}{5}$. Let $T^*>0$ be the maximal time $T$ such that $u$ verifies~\eqref{hyp:regularity}, and assume that $T^*$ is finite. 
		Assume that the rescaled function $U$ defined by~\eqref{eq:scaling} verifies:
		\begin{equation*}
			\lim_{t\to T^*} \|\mathrm{curl} \ U\|_{L^\infty(\tilde{\Omega}_t)}>0.
		\end{equation*}
		Then 
		\begin{equation*}
			{\Lambda}_p(t) \underset{t\to T^*}{\longrightarrow} \infty.
		\end{equation*}
	\end{cor}
	While the instability result Theorem~\ref{thm:bound_omega_lambda} is stated in the 3D variables $x\in \Omega$, the Corollary~\ref{cor} use the 2D set of toroidal variables $\tilde{x}\in \tilde\Omega$. The choice is arbitrary. However, the conditions on parameters are dependent on the choice of representation when $p<\infty$.

\vskip0.3cm

	This result provides scalings for which blow-up profiles, if they exist, become themselves unstable. In particular, we compare the scalings obtained here with the scalings numerically computed in~\cite{luo_potentially_2014} in the last section. Note in comparison that~\cite{elgindi_stability_2019} provides a stability result of the blow-up profiles in very strong norms, but for less regular solutions.

\vskip0.3cm

	Following~\cite{vasseur_blow-up_2019} we want to compare the growth on the vorticity $\omega=\mathrm{curl} \ u$ and the growth on the instabilities $\lambda_{p,\sigma}$. 
	The vorticity $\omega = \curl u$ verifies
	\begin{equation}\label{eq:vortex}
		\dpt \omega + u\cdot\nabla \omega = \omega\cdot\nabla u.
	\end{equation}
	In the axisymmetric case, the vorticity can be written  as the sum of the poloidal component $\omega_\theta\, e_\theta$, and the toroidal component 	
	\begin{align*}
		\tilde{\omega} := \omega_r\,e_r + \omega_z\,e_z.
	\end{align*}
	Note that the poloidal component $\omega_\theta$ of the vorticity is known to be important for axisymmetric solutions. Actually the system~\eqref{eq:Euler} can be reduced to a system of two equations describing  the evolution of $(u_\theta, \omega_\theta)$ (see Section~\ref{sec:blow_criteria}):
	\begin{align*}
		\dpt(ru_\theta) + u\cdot\nabla(ru_\theta) &= 0
		\\
		\dpt\!\(\frac{\omega_\theta}{r}\) + u\cdot\nabla\!\(\frac{\omega_\theta}{r}\) &=  \,\frac{\partial_zu^2_\theta}{r^2}.
		\end{align*}
	Blow-up criteria based only on the poloidal component were obtained by Chae and Kim~\cite{chae_breakdown_1996}, and Chae~\cite{chae_remarks_2005}.
		
\vskip0.2cm

	However, because the perturbations are themselves axisymmetric, we can only compare $\lambda_{p,\sigma}$ to the toroidal component of the vorticity $\tilde\omega$. A first step is then to obtain the following blow-up criteria \`a la Beale Kato Majda involving only the toroidal component of the vorticity. 
	
	\begin{prop}[Blow-up criterion]\label{thm:criterion}
		Assume $u$ is an axisymmetric solution of~\eqref{eq:Euler} verifying Hypothesis~\eqref{hyp:regularity} with initial condition $u^\init$ verifying $r\,u^\init \in L^\infty(\Omega)\cap L^q(\Omega)$ with $q<\frac{6}{5}$. Then $u$ can be extended as a solution verifying~\eqref{hyp:regularity} on a bigger interval of time if and only if
		\begin{align*}
			\int_0^T \left(\Nrm{\tilde{\omega}}{L^\infty(\Omega)} + \Nrm{\frac{\tilde{\omega}}{\sqrt{r}}}{L^\infty(\Omega)}^2\right) \d t < \infty.
		\end{align*}
	\end{prop}
	
	\begin{remark}\label{rmk:criterion_extended}
		One can replace the above criterion by the following more general one: $\tilde{\omega}\in L^1((0,T),L^\infty)$ and
		\begin{align*}
			\frac{\omega_r}{r^a}\in L^{1+\vartheta}((0,T), L^\infty) \text{ and } \frac{\omega_z}{r^b} \in L^{1+\vartheta}((0,T), L^\infty),
		\end{align*}
		where $(a,b)\in[0,s-\tfrac{5}{2})^2$ are such that $a+b\geq 1$ and $a\leq 2$, $\vartheta = \tfrac{3-a}{b+2} \in [0,1]$ and $q< 3\,\frac{b+a}{b+2}$ (see Lemma~\ref{lem:reverse_criterion} and \ref{lem:criterion_whole_space}). If one removes the upper bound $s-\frac{5}{2}$, then this criterion is still sufficient to extend the solution.
	\end{remark}

\section{Set up of the method}

	To simplify the notations, we will not denote the domain for the functions spaces when this domain is $\Omega$, so that for example $L^p := L^p(\Omega)$.
	
	As in~\cite{vasseur_blow-up_2019} and~\cite{vishik_spectrum_1996}, we use a WKB approximation of the solution $v$ of the linearized problem~\eqref{eq:linear_Euler} under the form
	\begin{equation*}
		v(t,x) \simeq b\!\(t,x\) e^{i S(t,x)/\eps} \text{ when } \eps\to 0,
	\end{equation*}
	with $b$ and $\xi = \nabla S$ verifying the bi-characteristic amplitude ODE system
	\begin{equation}\label{eq:ODE}
		\left\{\begin{array}{rcl}
		\dot{x_t} &=& u_t
		\\
		\dot{\xi_t} &=& -(\nabla u)_t\cdot \xi_t
		\\
		\dot{b_t} &=& - b_t \cdot (\nabla u)_t + 2 \, \frac{b_t \cdot (\nabla u)_t\cdot\xi_t}{\n{\xi_t}^2}\, \xi_t,
		\end{array}\right.
	\end{equation}
	where the subscript indicates that we look at the quantity along the flow, i.e. the notations have the following signification: $u_t = u(t,x_t)$, $(\nabla u)_t = \nabla u(t,x_t)$, $\xi_t = \xi_t(x_0,\xi_0) = \xi(t,x_t)$ and $b_t = b_t(x_0,\xi_0,b_0) = b(t,x_t)$. Remark that $x_t$ depends only $x_0$ while $\xi_t$ depends both on $x_0$ and $\xi_0$ and $b_t$ depends on $(x_0,\xi_0,b_0)$.
	The incompressibility condition on $v$ corresponds to $b_t\cdot\xi_t = 0$ in this approximation, which is a property conserved by~\eqref{eq:ODE}.\\
	
	 We want to obtain results similar as in the paper~\cite{vasseur_blow-up_2019} but in the axisymmetric case. In this setting, notice that since $S$ is invariant by rotation around the axis $z=0$, we obtain in polar coordinates $S(t,x) = S(t,r,z)$, so that $\dpth S = \xi_\theta = 0$, $\xi = \tilde{\xi}$, and the quantity $\beta(T)$ defined in~\cite{vasseur_blow-up_2019} is no more able to control $\omega_\theta$ with the constraint $\xi=\tilde{\xi}$. This is the reason why we first need Proposition~\ref{thm:criterion}, which is proved by proving that that the bounds on $\tilde{\omega}$ implies bounds on $\omega_\theta$ and then using Beale-Kato-Majda criterion. Then, we define
	 \begin{equation*}
	 	\beta_\sigma(t) := \sup_{\substack{(x_0,b_0,\tilde{\xi}_0)\in \Omega\times\R^3\times\SS^1\\ b_0\cdot\xi_0 = 0,\, \n{b_0} = r_0^{\sigma}}} \n{r_t^{-\sigma}\,b_t(x_0,\tilde{\xi}_0,b_0)}.
	 \end{equation*}
	 where $\tilde{x}_t = (r_t,z_t)$. This quantity controls two components of the vorticity in the sense of the following proposition.
	
	\begin{prop}\label{prop:bound_omega_axi}
		Assume $u$ is an axisymmetric solution of~\eqref{eq:Euler} verifying Hypothesis~\eqref{hyp:regularity}. Then for any $T\in (0,T^*)$ and any $\sigma\in\R$, we have
		\begin{align*}
			\Nrm{\frac{\tilde{\omega}(T,\cdot)}{r^a}}{L^\infty} &\leq \Nrm{\tilde{\omega}^\init}{L^\infty} \beta_\sigma(T)^{2+a}.
		\end{align*}
	\end{prop}
	
	Moreover, this quantity can be controlled by the norm of the semigroup  corresponding to the linearized Euler equation~\eqref{eq:linear_Euler} in weighted Lebesgue spaces.
	
	\begin{prop}\label{prop:bound_beta_axi}
		Let $T\in(0,T^*)$, $p\in[1,\infty)$, $\sigma\in\(-\,\frac{2}{p'},\frac{2}{p}\)$ and assume $u$ is an axisymmetric solution of~\eqref{eq:Euler} verifying Hypothesis~\eqref{hyp:regularity}. Then
		\begin{align*}
			\beta_{\sigma}(T)\leq \lambda_{p,\sigma}(T).
		\end{align*}
	\end{prop}
	
	Combining these two propositions and the new blow-up criterion leads to the proof of the main Theorem.
	\begin{proof}[Proof of Theorem~\ref{thm:bound_omega_lambda}]
		Combining Proposition~\ref{prop:bound_omega_axi} and Proposition~\ref{prop:bound_beta_axi} with $a > 0$ or $a=0$ we deduce that
		\begin{align*}
			\Nrm{\frac{\tilde{\omega}(T,\cdot)}{r^a}}{L^\infty} &\leq \Nrm{\tilde{\omega}^\init}{L^\infty} \beta_{\sigma}(T)^{2+a} \leq \Nrm{\tilde{\omega}^\init}{L^\infty} \lambda_{p,\sigma}(T)^{2+a}
			\\
			\Nrm{\tilde{\omega}(T,\cdot)}{L^\infty} &\leq \Nrm{\tilde{\omega}^\init}{L^\infty} \beta_\sigma(T)^2 \leq \Nrm{\tilde{\omega}^\init}{L^\infty} \lambda_{p,\sigma}(T)^2.
		\end{align*}
		and we can then conclude by Remark~\ref{rmk:criterion_extended} with $a = b \in[\frac{1}{2},s-\frac{5}{2})$.
	\end{proof}

\section{Blow-up criterion}\label{sec:blow_criteria}

	We  use the same notation $u$ both for the function of the variable $x\in\R^3$ and for the function of the variable $\tilde{x} = (r,z) \in \tilde{\Omega}$. We will however use the notation $\nabla$ as the gradient with respect to $x$, while we will define $\tilde{\nabla} := (\dpr,\dpz)$. Remark however that for any axisymmetric scalar function $w(x) = w(r,z) \in \R$ it holds
	\begin{align*}
		\nabla w = \nabla \(w(r,z)\) = \dpr w\, e_r + \dpz w\, e_z = \tilde{\nabla} w.
	\end{align*}
	We will write $\tilde{u} = u_r\,e_r + u_z e_z$. Remark that since $u$ is axisymmetric, we have $\Div{u_\theta e_\theta} = \dpth u_\theta = 0$, therefore $\tilde{u}$ is also divergence free and with these notations
	\begin{align*}
		\tilde{\divg}(r\tilde{u}) = \tilde{\divg}(\tilde{u}) + \frac{u_r}{r} = \Div{\tilde{u}} = 0,
	\end{align*}
	Moreover, remarking that $u_r$ and $u_\theta$ are odd functions of $r$, if $u$ is continuous, we deduce that $u_r(0,z) = u_\theta(0,z) = 0$. If $u\in H^{\frac{5}{2}+\eps}\subset C^1$, then we have even better. Computing the gradient in cylindrical coordinates yields
	\begin{align}\label{eq:bound_cylindric_nabla}
		 \n{\dpr u_r}^2 + \n{\dpr u_\theta}^2 + \n{\frac{u_r}{r}}^2 +\n{\frac{u_\theta}{r}}^2 &\leq \n{\nabla u}^2 < \infty.
	\end{align}
	If $u\in H^{\frac{7}{2}+\eps} \subset C^2$, computing the Hessian matrix $\nabla^2 u$ in cylindrical coordinates yields
	\begin{align}\label{eq:bound_cylindric_nabla_2}
		 \n{\dpr^2u_r}^2 + \n{\dpr^2u_\theta}^2 + \n{\frac{\dpr u_r}{r}}^2 + \n{\frac{\dpr u_\theta}{r}}^2 + \n{\frac{u_r}{r^2}}^2 +\n{\frac{u_\theta}{r^2}}^2 &\leq \n{\nabla^2 u}^2 < \infty.
	\end{align}
	
	From these estimates we deduce the following bounds
	\begin{lem}\label{lem:reverse_criterion}
		Assume $u\in H^s$ with $s>\frac{5}{2}$ is axisymmetric and let $\omega = \curl u$. Then for any $(a,b)\in \lt[0,s-\frac{5}{2}\rt)^2$, there exists a constant $C>0$ such that
		\begin{align*}
			\Nrm{\frac{\omega_r}{r^a}}{L^\infty} \leq C \Nrm{u}{H^s} \text{ and } \Nrm{\frac{\omega_z}{r^b}}{L^\infty} \leq C \Nrm{u}{H^s}.
		\end{align*}
	\end{lem}
	
	\begin{proof}
		As a consequence of the Sobolev embedding $H^{\frac{5}{2}+\eps}\subset C^1$, $H^{\frac{7}{2}+\eps} \subset C^2$, the fact that $\dpz u_\theta = -\omega_r$ and  and the above inequalities~\eqref{eq:bound_cylindric_nabla} and~\eqref{eq:bound_cylindric_nabla_2}, we obtain for any $\eps>0$ the existence of a constant $C>0$ such that
		\begin{align*}
			\Nrm{\omega_r}{L^\infty} \leq C \Nrm{u}{H^{\frac{5}{2}+\eps}} \text{ and } \Nrm{\frac{\omega_r}{r}}{L^\infty} \leq C \Nrm{u}{H^{\frac{7}{2}+\eps}},
		\end{align*}
		and we conclude by interpolation. Since $\omega_z = \dpr u_\theta + \frac{u_\theta}{r}$, the same reasoning implies the result for $\omega_z$. When $s>7/2$, one can use higher derivatives. Remark that the result also follows from Hardy's and Sobolev's inequalities.
	\end{proof}

	\begin{lem}\label{lem:bound_om}
		Let $p\geq 1$, and $(a,b)\in\R^2$ verifying $\(a+b-1\)\(3-a\)\geq 0$ and $b\neq \frac{2}{p}-1$. Assume $u$ is an axisymmetric solution of~\eqref{eq:Euler}  verifying Hypothesis~\eqref{hyp:regularity} with initial condition $u^\init$ such that $r\,u^\init\in L^{p\(1-\vartheta\)}$ with $\vartheta = \frac{3-a}{b+2}$. Then for any $T\in(0,T^*)$, it holds
		\begin{align*}
			\Nrm{\frac{\omega_\theta(T)}{r}}{L^p} \leq \Nrm{\frac{\omega_\theta^\init}{r}}{L^p} + C \int_0^T \Nrm{\frac{\omega_r}{r^a}}{L^\infty}\Nrm{\frac{\omega_z}{r^b}}{L^\infty}^{\vartheta} \d t,
		\end{align*}
		where $C = \frac{2}{|b+1-d/p|} \Nrm{r\,u^\init}{L^{p\(1-\vartheta\)}}^{1-\vartheta}$.
	\end{lem}
	
	\begin{remark}\label{rmk:criterion_bounded_domen}
		In the case when $\Omega$ is a bounded domain and $\frac{\omega_r}{r^a}$ and $\frac{\omega_z}{r^b}$ are in $L^{1+\vartheta}((0,T),L^\infty)$, we deduce in particular that
		\begin{align*}
			\int_0^T \Nrm{\omega}{L^\infty} \d t < C
		\end{align*}
		for some constant $C>0$ depending on the size of the domain.
	\end{remark}
	
	\begin{proof}
		Remarking that the pressure $P$ does not depend on $\theta$ since $u$ is axisymmetric and taking the $\theta$ coordinate of the first equation in~\eqref{eq:Euler} and of Equation~\eqref{eq:vortex} yields
		\begin{align*}
			\dpt u_\theta + u_r\dpr u_\theta + u_z \dpz u_\theta + u_\theta \frac{u_r}{r} &= 0
			\\
			\dpt \omega_\theta + u_r\dpr \omega_\theta + u_z\dpz\omega_\theta + \omega_r \frac{u_\theta}{r} &= \omega_r\dpr u_\theta + \omega_z\dpz u_\theta + \omega_\theta \frac{u_r}{r},
		\end{align*}
		which can be rewritten
		\begin{align*}
			\dpt u_\theta + \frac{u_r}{r}\dpr(ru_\theta) + u_z \dpz u_\theta &= 0
			\\
			\dpt \omega_\theta + r u_r\dpr\(\frac{\omega_\theta}{r}\) + u_z\dpz\omega_\theta &= - 2\,\frac{\omega_r\,u_\theta}{r},
		\end{align*}
		where we used the fact that $\dpz u_\theta = -\omega_r$ and $\dpr u_\theta = \omega_z - \frac{u_\theta}{r}$. Therefore, defining $\tilde{\omega} = (\omega_r,\omega_z)$, multiplying the first equation by $r$ and dividing the second equation by $r$ we get
		\begin{align}\label{eq:nu}
			\dpt(ru_\theta) + u\cdot\nabla(ru_\theta) &= 0
			\\\label{eq:om}
			\dpt\!\(\frac{\omega_\theta}{r}\) + u\cdot\nabla\!\(\frac{\omega_\theta}{r}\) &=  - 2\,\frac{\omega_r\,u_\theta}{r^2}.
		\end{align}
		We immediately deduce from Equation~\eqref{eq:nu} that for any $q\in[1,\infty]$,
		\begin{equation*}
			\Nrm{ru_\theta}{L^q} = \Nrm{ru_\theta^\init}{L^q}.
		\end{equation*}
		Since $u_\theta(0,z)=0$, we can use Hardy's inequality (see e.g. \cite[Equation~(1.3.3)]{mazya_sobolev_2011}) which tells that for any $p\in(1,\infty)$ and $\sigma \neq \frac{d}{p}$ and function $\nu : \R^d\to \R$ with the additional assumption that $\nu(0)=0$ if $\sigma>\frac{d}{p}$, it holds
		\begin{align*}
			\Nrm{\frac{\nu(y)}{\n{y}^{\sigma}}}{L^p(\R^d)} \leq \frac{1}{\n{\sigma-d/p}} \Nrm{\frac{\nabla \nu(y)}{\n{y}^{\sigma-1}}}{L^p(\R^d)}.
		\end{align*}
		In particular, taking $\nu(r) = ru_\theta(r,z)$, remarking that $\dpr(ru_\theta) = r\omega_z$, and then taking the $L^p$ norm with respect to $z$, we obtain for any $p\in[1,\infty]$ (by passing to the limit to get $p=\infty$ and $p=1$) and any $b\neq \frac{2}{p}-1$,
		\begin{equation}\label{eq:Hardy}
			\Nrm{\frac{u_\theta}{r^{b+1}}}{L^p} \leq C_{d,b,p} \Nrm{\frac{\omega_z}{r^{b}}}{L^p}.
		\end{equation}
		Therefore, from Equation~\eqref{eq:om}, we get for any $\vartheta\in[0,1]$
		\begin{align*}
			\dt \Nrm{\frac{\omega_\theta}{r}}{L^p} &\leq 2 \Nrm{\frac{\omega_r}{r^a}}{L^\infty} \Nrm{\frac{u_\theta}{r^{2-a}}}{L^p}
			\\
			&\leq 2 \Nrm{\frac{\omega_r}{r^a}}{L^\infty} \Nrm{\frac{u_\theta^{\vartheta}}{r^{3-a-\vartheta}}}{L^\infty} \Nrm{r\,u_\theta}{L^{p\(1-\vartheta\)}}^{1-\vartheta}
			\\
			& \leq \tfrac{2}{|b+1-2/p|} \Nrm{\frac{\omega_r}{r^a}}{L^\infty} \Nrm{\frac{\omega_z}{r^b}}{L^\infty}^{\vartheta}\Nrm{r\,u_\theta^\init}{L^{p\(1-\vartheta\)}}^{1-\vartheta},
		\end{align*}
		with $b = \frac{3-a-2\,\vartheta}{\vartheta}$. This implies the result by Gronwall's inequality.
	\end{proof}
	
	When $\Omega = \R^3$, we still have to get estimates on the $L^\infty$ norm of $\omega$ for large values of $r$.
	\begin{lem}\label{lem:criterion_whole_space}
		Let $T>0$, $\Omega :=\R^3$ and $(a,b)\in\R^2$ verifying either $b> -1$ and $a\in[1-b,2]$ or $3\leq a\leq -b$. Assume $u$ is an axisymmetric solution of~\eqref{eq:Euler}  verifying Hypothesis~\eqref{hyp:regularity} with initial condition $u^\init$ such that $r\,u^\init\in L^\infty\cap L^q$ with $q<3\,\frac{b+a}{b+2}$ and
		\begin{align*}
			\frac{\omega_r}{r^a}\in L^{1+\vartheta}((0,T), L^\infty) \text{ and } \frac{\omega_z}{r^b} \in L^{1+\vartheta}((0,T), L^\infty),
		\end{align*}
		with $\vartheta = \max(\tfrac{2-a}{b+2},\tfrac{3-a}{b+2}) \in [0,1]$. Then it holds
		\begin{align*}
			\omega_\theta\in L^\infty((0,T)\times\R^d).
		\end{align*}
	\end{lem}
	
	\begin{proof}
		We come back to Equation~\eqref{eq:om} to obtain that the equation for $\omega_\theta$ can be written
		\begin{align*}
			\dpt \omega_\theta + u\cdot\nabla\omega_\theta &= \frac{\omega_\theta\, u_r}{r} - 2\,\frac{\omega_r\,u_\theta}{r}.
		\end{align*}
		Hence, similarly as in the proof of Lemma~\ref{lem:bound_om} and using our new bound on $\frac{\omega_\theta}{r}$, we can control the $L^p$ norm in the following way
		\begin{align}\nonumber
			\dt\Nrm{\omega_\theta}{L^p} &\leq  \Nrm{\frac{\omega_\theta}{r}}{{L^p}} \Nrm{u_r}{{L^\infty}} + 2\Nrm{\frac{\omega_r}{r^a}}{{L^\infty}} \Nrm{\frac{u_\theta}{r^{1-a}}}{{L^p}}
			\\\label{eq:om_th_Lp}
			&\leq \Nrm{\frac{\omega_\theta}{r}}{{L^p}} \Nrm{u_r}{{L^\infty}} + C\Nrm{\frac{\omega_r}{r^a}}{{L^\infty}} \Nrm{\frac{\omega_z}{r^b}}{{L^\infty}}^{\vartheta_2} \Nrm{r u_\theta^\init}{{L^{p\(1-\vartheta_2\)}}}^{1-\vartheta_2},
		\end{align}
		with $\vartheta_2 = \frac{2-a}{b+2} \in [0,1]$. To bound $u_r$, we use the Biot and Savart law which gives us
		\begin{align*}
			u(x) = \curl\int_{\R^3}K(x-y)\,\omega(y)\dd y = \curl K*\omega,
		\end{align*}
		where $K(x) = \frac{1}{4\pi\n{x}}$ is the Newtonian kernel. Taking the scalar product with $e_r$ and using the fact that since we are in an axisymmetric setting $\partial_\theta\omega_z = 0$, we obtain
		\begin{align*}
			u_r = -\dpz K*\omega_\theta,
		\end{align*}
		and since $\dpz K \in L^{3/2,\infty}$, we deduce
		\begin{align*}
			\Nrm{u_r}{L^\infty} &\leq C_K\Nrm{\omega_\theta}{L^{3,1}} \leq C\Nrm{\omega_\theta}{L^{3-\eps}\cap L^{3+\eps}}.
		\end{align*}
		Therefore, by taking $p=p_0\in(2,3)$ and $p=p_1>3$ in Inequality~\eqref{eq:om_th_Lp}, by Gronwall's lemma we deduce that $\omega_\theta \in L^\infty((0,T),L^{p_0}\cap L^{p_1})$ and
		\begin{align*}
			\Nrm{u_r}{L^\infty} &\leq C\Nrm{\omega_\theta}{L^{p_0}\cap L^{p_1}} \leq C_{T,\tilde{\omega}},
		\end{align*}
		where $C_{T,\tilde{\omega}}$ depends only on $T$, on the initial conditions and on the bounds on $\Nrm{\frac{\omega_r}{r^a}}{{L^\infty}}$ and $ \Nrm{\frac{\omega_z}{r^b}}{{L^\infty}}$. Now that we know that $u_r\in L^\infty((0,T)\times\R^3)$, we can take $p=\infty$ in Inequality~\eqref{eq:om_th_Lp} and then conclude by Gronwall's inequality.
	\end{proof}

\section{Proof of the Linear instability}

\subsection{Control of the vorticity by \texorpdfstring{$\beta$}{beta}}

	\begin{proof}[Proof of Proposition~\ref{prop:bound_omega_axi}]
		Fix $(T,x_T)\in(0,T^*)\times\Omega\backslash\{r=0\}$ and define backward the solution $x_t$ of the first equation in the system~\eqref{eq:ODE} for $t\in[0,T]$, and $\omega_t = \omega(t,x_t)$. Taking $\xi'_T$ as the unit vector such that $(\xi'_T)_\theta = 0$ and $\tilde{\xi}'_T\cdot\tilde{\omega}_T = \n{\tilde{\omega}}_T$, since $\xi'$ solves the backward dual vorticity equation, it holds $\xi'_0\cdot\omega_0 = \xi'_T\cdot\omega_T$ (see~\cite[Equation~(8)]{vasseur_blow-up_2019}). Since $\xi'_t$ remains axisymmetric, $(\xi'_0)_\theta = 0$, so that
		\begin{equation*}
			\tilde{\xi}'_0\cdot\tilde{\omega}_0 = \xi'_0\cdot\omega_0 = \xi'_T\cdot\omega_T = \tilde{\xi}'_T \cdot\tilde{\omega}_T = \n{\tilde{\omega}_T}.
		\end{equation*}
		Thus, with the notation $r_t = r_t(x_0) := (x_t)_r$, we get the following inequality
		\begin{equation}\label{eq:xi_0_vs_omega}
			r_T^{-a} \n{\tilde{\omega}_T} = r_T^{-a} \(\tilde{\xi}'_0\cdot\tilde{\omega}_0\) \leq r_0^{a}\,r_T^{-a} \n{\xi'_0}\Nrm{r^{-a}\omega(0,\cdot)}{L^\infty}.
		\end{equation}
		Now consider vectors $(b'_0,b''_0,\xi'_0)\in (\R^3)^2\times\SS^1$ such that $b'_0$, $b''_0$ and $\xi'_0$ are orthogonal to each others, $\n{b'_0} = \n{b''_0} = r_0^{\sigma}$ and $r_0^{-2\sigma}\(b'_0\times b''_0\)\cdot\xi'_0 = \n{\xi'_0}$, and let $b'_t$ and $b''_t$ be the solutions of~\eqref{eq:ODE} corresponding to the same $\xi'_t$ defined before. Then by~\cite[Equation~(10)]{vasseur_blow-up_2019}, we get
		\begin{equation}\label{eq:xi_0_vs_b}
			r_0^{2\sigma}\n{\xi'_0} = \(b'_0\times b''_0\)\cdot\xi'_0 = \(b'_T\times b''_T\)\cdot\xi'_T \leq \n{b'_T}\n{b''_T}.
		\end{equation}
		We can also consider an other solution $\xi_t = \tilde{\xi}_t$ of equation~\eqref{eq:ODE} such that $\xi_0\cdot\omega_0 = 0$. Once again, since this scalar product is conserved, for any $t\in[0,T]$ we still have $\xi_t\cdot\omega_t = 0$. Now define $b_t$ as a solution of the third equation in~\eqref{eq:ODE} such that $b_0 = r_0^\sigma e_\theta$ and the incompressibility condition $\xi\cdot b = 0$ is satisfied. For such a solution $b$, since $\xi_\theta=0$ we have
		\begin{align*}
			\dpt \!\(r b_\theta\)_t &= (u_r b_\theta)_t + r_t\(-b\cdot\nabla u\cdot e_\theta -\frac{u_\theta}{r}\,b_r \)_t = - \(r b\cdot\nabla u_\theta\)_t -\(u_\theta b_r\)_t.
		\end{align*}
		Thus, using the fact that $\dpz u_\theta = -\omega_r$ and $\dpr u_\theta = \omega_z - \frac{u_\theta}{r}$, we obtain
		\begin{align*}
			\dpt \!\(r b_\theta\)_t & = \(b\times \omega\)_t\cdot r_t e_\theta.
		\end{align*}
		However, $b\times \omega$ has to be parallel to $\xi$ since $\xi\cdot b = 0$ and $\xi\cdot\omega = 0$, and since $\xi_\theta = 0$, we deduce that the right-hand side of the above equation is $0$, and so $\(r b_\theta\)_t$ is constant. Therefore,
		\begin{equation}\label{eq:rt_vs_b}
			r_0^{\sigma+1} = (r b_\theta)_0 = (r b_\theta)_T \leq r_T \n{b_T}.
		\end{equation}
		Combining Inequality~\eqref{eq:xi_0_vs_omega}, Inequality~\eqref{eq:xi_0_vs_b} and Inequality~\eqref{eq:rt_vs_b}, we get
		\begin{align*}
			r_T^{-a} \n{\tilde{\omega}_T} &\leq r_0^{-\sigma(a+2)}\(r_0^{\sigma+1}r_T^{-1}\)^{a} r_0^{2\sigma}\n{\xi'_0}\Nrm{r^{-a}\omega(0,\cdot)}{L^\infty}
			\\
			&\leq r_0^{-\sigma(a+2)} \n{b_T}^{a} \n{b'_T}\n{b''_T} \Nrm{r^{-a}\omega(0,\cdot)}{L^\infty}
			\\
			&\leq \beta_\sigma(T)^{2+a} \Nrm{r^{-a}\omega(0,\cdot)}{L^\infty}.
		\end{align*}
		which leads to the result.
	\end{proof}

\subsection{Control of \texorpdfstring{$\beta$}{beta} by the linear stability}

	Before proving Proposition~\ref{prop:bound_beta_axi}, we need the following stability result of the linearized Euler equation with additional source term in weighted $L^p$ spaces. 
	\begin{lem}\label{lem:weighted_stability}
		Let $T>0$, $p\in[1,\infty]$, $\alpha\in\(-\,\frac{2}{p}, \frac{2}{p'}\)$ and assume $u$ is a solution of~\eqref{eq:Euler} verifying Hypothesis~\eqref{hyp:regularity}. Then there exists a unique $v \in C^0([0,T),H^1(\Omega))$ solution of the system
		\begin{equation*}
			\begin{array}{rll}
				\dpt v + u\cdot\nabla v + v\cdot \nabla u + \nabla Q &=& f
				\\
				\Div{v} &=& 0,
			\end{array}
		\end{equation*}
		with boundary conditions $v\cdot n = 0$ on $\partial\Omega$ and initial condition $v(0,\cdot) = v^\init$. Moreover for any $t\in[0,T]$,
		\begin{align*}
			\Nrm{r^\alpha v}{L^p} \leq e^{U_T}\(\Nrm{r^\alpha v^\init}{L^p} + \Nrm{r^\alpha f}{L^1((0,T),L^p)}\),
		\end{align*}
		with $U_T = \int_0^T \Nrm{u(t,\cdot)}{H^s}\dd t$.
	\end{lem}
	
	\begin{proof}
		The existence and uniqueness of $v$ follows from~\cite{inoue_existence_1979}. Remarking that
		\begin{align*}
			\dpt\!\(r^\alpha v\) + u\cdot\nabla \(r^\alpha v\) =  \(r^\alpha v\)\cdot \(\alpha\, \frac{u_r}{r} - \nabla u\) + r^\alpha\(f-\nabla Q\),
		\end{align*}
		and then taking the derivative of the $L^p$ norm yields
		\begin{align*}
			\dpt\!\Nrm{r^\alpha v}{L^p} \leq \Nrm{\alpha\, \frac{u_r}{r} - \nabla u}{L^\infty}\Nrm{r^\alpha  v}{L^p} + \Nrm{r^\alpha \(f-\nabla Q\)}{L^p}.
		\end{align*}
		By Inequality~\eqref{eq:bound_cylindric_nabla}, we have $\Nrm{\alpha\, \frac{u_r}{r} - \nabla u}{L^\infty} \leq \(\n{\alpha}+1\)\Nrm{\nabla u}{L^\infty}$. Moreover, since $v\cdot\nabla u\cdot n = 0$, the pressure $Q$ verifies the Neumann problem
		\begin{align*}
			&&-\Delta Q &= \Div{2\,v\cdot\nabla u - f} &&(\text{in } \Omega) &&
			\\
			&&\partial_n Q &= \(2\,v\cdot\nabla u - f\)\cdot n &&(\text{on } \partial\Omega), &&
		\end{align*}
		so that by elliptic regularity theory (see e.g.~\cite[Theorem~1.6]{yang_weighted_2020}) we have
		\begin{align*}
			\Nrm{r^\alpha \nabla Q}{L^p} &\lesssim \Nrm{r^\alpha\(2\,v\cdot\nabla u - f\)}{L^p}
			\\
			&\lesssim 2 \Nrm{r^\alpha v}{L^p} \Nrm{\nabla u}{L^\infty} + \Nrm{r^\alpha f}{L^p},
		\end{align*}
		as soon as $r^{\alpha p}$ is in the Muckenhoupt class $A_p$. One can check (see e.g.~\cite[Theorem~1.1]{dyda_muckenhoupt_2019}) that this holds as soon as $2\(1-p\) < -\alpha p < 2$. Hence, the estimates follows from the Sobolev embedding $H^s\subset W^{1,\infty}$ for $s>\frac{5}{2}$ and Gronwall's inequality.
	\end{proof}

	With this result, we are now ready to get a bound on $\beta$ in weighted Lebesgue spaces.
	
	\begin{proof}[Proof of Proposition~\ref{prop:bound_beta_axi}]
		Since $u\in C^0([0,T],C^1)$, as in~\cite[Proposition~1]{vasseur_blow-up_2019}, we can define the flow $\gamma$ associated to $u$ verifying $\dpt \gamma(t,x) = u(t,\gamma(t,x))$ with initial condition $\gamma(0,x) = x$. It is a Lipschitz function of time and is $C^2$ with respect to the initial condition. The inverse of $\gamma(t,\cdot)$, denoted $\gamma^{-1}(t,\cdot)$, has the same regularity.
		
		Now let $(x_t,\xi_t,b_t)$ be a solution of the ordinary differential equations system~\eqref{eq:ODE} with initial condition $(x_0,\xi_0,b_0) = (x_0,\tilde{\xi}_0,b_0)\in \Omega\times \SS^1\times \SS^2$ be such that $b_0\cdot\xi_0 = 0$, and define
		\begin{align*}
			\xi(t,x) &:= \xi_t(\tilde{\gamma}^{-1}(t,x), \xi_0)
			\\
			b(t,x) &:= b_t(\tilde{\gamma}^{-1}(t,x), \xi_0),
		\end{align*}
		or equivalently $\xi(t,\gamma(t,x)) = \xi_t(\tilde{x}, \xi_0)$ and $b(t,\gamma(t,x)) = b_t(\tilde{x}, \xi_0)$. In particular $x_t = \gamma(t,x_0)$, $\xi_t(x_0,\xi_0) = \xi(t,x_t)$ and $b_t(x_0,\xi_0,b_0) = b(t,x_t)$. Now we can also define $S$ as
		\begin{equation*}
			S(t,x) := \tilde{\gamma}^{-1}(t,x)\cdot\xi_0.
		\end{equation*}
		This implies that $S$ is axisymmetric and $S(t,\gamma(t,x)) = \tilde{x}\cdot\xi_0$, so that $S$ is a solution to
		\begin{equation}\label{eq:transport_S}
			\dpt S + u\cdot\nabla S = 0,
		\end{equation}
		with initial condition $S_0(x) = \tilde{x}\cdot \xi_0$. Differentiating the equation verified by $S$ yields
		\begin{equation*}
			\dpt \nabla S + u\cdot\nabla (\nabla S) =- \nabla u\cdot(\nabla S),
		\end{equation*}
		which is exactly the equation verified by $\xi$, with the same initial value. By uniqueness, we deduce that $\xi = \tilde{\xi} = \nabla S$.
		
		Since $u\in C^0([0,T],C^1)$, by Inequality~\eqref{eq:bound_cylindric_nabla} we deduce that $\frac{u_r}{r}\in L^\infty([0,T], L^\infty)$. Therefore, remarking that from the system~\eqref{eq:ODE} we have
		\begin{align*}
			\dpt(r_t^\alpha\,b_t) = \alpha\(\frac{u_r}{r} \)_t \(r_t^\alpha\,b_t\) - b_t \cdot (\nabla u)_t + 2 \, \frac{\(r_t^\alpha\,b_t\) \cdot (\nabla u)_t\cdot\xi_t}{\n{\xi_t}^2}\, \xi_t,
		\end{align*}
		we deduce that $r_t^\alpha\,b_t$ remains bounded on $[0,T]$ and $\beta_{-\alpha}(T) < \infty$. Therefore, for $\eta>0$, we can choose $(x_0,\xi_0,b_0)$ such that $r_0\neq 0$ and
		\begin{equation}\label{eq:bt_approx_beta_axi}
			\beta_{-\alpha}(T) \leq \(1+\eta\)r_T^\alpha\n{b_T(x_0,\xi_0,b_0)}.
		\end{equation}
		The regularity and the uniqueness of the flow $\gamma$ implies the existence of some constant $\delta>0$ such that the ball $\tilde{B}_\delta(x_T) \subset \R^2$ of center $\tilde{x}_T = \tilde{\gamma}_T(x_0)$ and radius $\delta$ is strictly included in $\tilde{\Omega}\setminus\{r=0\}$, and
		\begin{align}\label{eq:b_approx_bt_axi}
			\(1-\eta\)r_T(x_0)^\alpha\n{b_T(x_0,\xi_0,b_0)} &\leq \inf_{x\in \cA_\delta(x_T)} \n{r_T(\gamma_T^{-1}(x))^\alpha\, b_T(\gamma_T^{-1}(x),\xi_0,b_0)}
			\\\nonumber
			&\leq \inf_{x\in \cA_\delta(x_T)} \n{r^\alpha\, b(T,x)},
		\end{align}
		where $x = (r,z)$ and we defined the annulus $\cA_\delta(x_T) := \{x\in \Omega, \tilde{x}\in \tilde{B}_\delta(x_T)\}$. Now we define $\varphi_T$ as a smooth function supported in $\cA_\delta(x_T)$ and such that 
		\begin{align}\label{eq:phi_norm_axi}
			\Nrm{\varphi_T}{L^p} = 1,
		\end{align}
		and for any $t\in[0,T)$, $\varphi(t,x) := \varphi_T({\gamma_t^{-1}(x)})$. We also define
		\begin{align*}
			v_{\eps,\delta} &:= \eps \Curl{\frac{b\times \xi}{|\xi|^2}\,\varphi\,e^{iS/\eps}},
		\end{align*}
		so that $\divg v_{\eps,\delta} = 0$. Since $\varphi$ is compactly supported in $\Omega$, $v_{\eps,\delta}$ is also compactly supported in $\Omega$, and in particular, $v_{\eps,\delta}\cdot n = 0$ on $\partial\Omega$. Moreover, $v_{\eps,\delta}$ is axisymmetric, since it is the case for $b$, $S$ and $\varphi$.
		
		As proved in~\cite{vasseur_blow-up_2019}, for any $t\in[0,T]$, we still have the property of orthogonality $\xi\cdot b = 0$ and from this we deduce the following formula
		\begin{align}\label{eq:expand_v_axi}
			v_{\eps,\delta}
			&= i\,\varphi\,b\,e^{iS/\eps} + \eps\,c_{\eta,\delta}\,e^{iS/\eps},
		\end{align}
		with $c_{\eta,\delta} = \Curl{\frac{b \times \xi}{\n{\xi}^2}\,\varphi}$ independent of $\eps$. Remark that since $\nabla u\in L^\infty$, by the second equation of the ODE system~\eqref{eq:ODE} we deduce
		\begin{equation*}
			\n{\xi_t} \geq e^{-t\,\Nrm{\nabla u}{L^\infty\!\([0,T]\times\Omega\)}} \n{\xi_0} \geq C_{T},
		\end{equation*}
		where $C_{T} = e^{-T\,\Nrm{\nabla u}{L^\infty\!\([0,T]\times\Omega\)}} > 0$. This implies also a bound of the $x$ dependent function $\n{\xi}^{-1}$ on $\cA_\delta(x_t)$ by choosing $\delta$ sufficiently small, since this function is $C^2$ with respect to its initial conditions, from which we deduce $C_{\eta,\delta} := \Nrm{r^\alpha c_{\eta,\delta}}{L^p} < \infty$, and 
		\begin{equation}\label{eq:WKB_vs_conservative}
			\n{\Nrm{\varphi\, r^\alpha\, b}{L^p} - \Nrm{v_{\eps,\delta}\,r^\alpha}{L^p}} \leq \eps\,C_{\eta,\delta}.
		\end{equation}
		This inequality, combined with the fact that $r_0^\frac{a}{2}b_0 = 1$ and $\Nrm{\varphi_T}{L^p} =1$, yields
		\begin{equation*}
			Z_\eps := \Nrm{r^\alpha v_{\eps,\delta}^\init}{L^p} \leq 1+ \eps\,C_{\eta,\delta}.
		\end{equation*}
		Now by inequalities~\eqref{eq:bt_approx_beta_axi},~\eqref{eq:b_approx_bt_axi}, the normalization condition~\eqref{eq:phi_norm_axi} and the fact that $\varphi$ is supported in $\cA_\delta(x_T)$, and by Inequality~\eqref{eq:WKB_vs_conservative}, we obtain
		\begin{align*}
			\beta_{-\alpha}(T) &\leq \tfrac{1+\eta}{1-\eta} \inf_{\cA_\delta(x_T)} \n{r^\alpha\,b(T,\cdot)}
			\\
			&\leq \tfrac{1+\eta}{1-\eta} \Nrm{(\varphi\,r^\alpha\,b )(T,\cdot)}{L^p}
			\\
			&\leq \tfrac{1+\eta}{1-\eta} \(\Nrm{r^\alpha v_{\eps,\delta}(T,\cdot)}{L^p} + \eps\,C_{\eta,\delta}\)
			\\
			&\leq \tfrac{1+\eta}{1-\eta} \(Z_\eps\Nrm{r^\alpha v(T,\cdot)}{L^p} + \Nrm{r^\alpha\(v_{\eps,\delta} - Z_\eps v\)(T,\cdot)}{L^p} + \eps\,C_{\eta,\delta}\),
		\end{align*}
		where $v$ is a solution of the linearized Euler equation~\eqref{eq:linear_Euler} with initial condition $v^\init = Z_\eps^{-1}v_{\eps,\delta}^\init$. Moreover, as in~\cite{vasseur_blow-up_2019}, it holds
		\begin{equation*}
			\dpt v_{\eps,\delta} + u\cdot\nabla v_{\eps,\delta} + v_{\eps,\delta}\cdot\nabla u + \nabla q_{\eps,\delta} = \eps\,R_{\eps,\delta},
		\end{equation*}
		where using the fact that $\varphi$ is compactly supported in a ball not containing the central axis, $\Nrm{r^\alpha R_{\eps,\delta}}{L^p} < C_{\eta,\delta}$ for some constant $C_{\eta,\delta}$ independent of $\eps$. Thus, applying Lemma~\ref{lem:weighted_stability} to $v_{\eps,\delta}-Z_\eps v$, we arrive at
		\begin{align*}
			\beta_{-\alpha}(T) &\leq \tfrac{1+\eta}{1-\eta} \(Z_\eps \Nrm{r^\alpha v(T,\cdot)}{L^p} + \eps\, C_{\eta,\delta}\,e^{U_T} + \eps\,C_{\eta,\delta}\),
		\end{align*}
		and we obtain the result by letting $\eps$ go to $0$ and then $\eta$ go to $0$.
	\end{proof}

\section{Proof of Corollary~\ref{cor}}
	
	We want to prove here that if we know that the solution is blowing up in a self-similar way, then the system is linearly unstable even if we are in the appropriate scale. Hence, we assume that there exists $T>0$ such that the solution $u^\circ$ of the Euler equations~\eqref{eq:Euler} is of the form~\eqref{eq:scaling}. If one would want to observe the shape of $U^\circ$, one could scale the amplitude by $(T-t)^\alpha$ and scale the positions using the new variable $y = \frac{\bar{x}-\bar{x}^\circ}{(T-t)^\beta}$. At this scale, perturbations take of the form
	\begin{equation*}
		U(t, y) := U^\circ\!\(t,y\) + V(t,y).
	\end{equation*}
	Equivalently, this defines a function
	\begin{equation*}
		u(t,x) = u^\circ(t,x) + v(t,x),
	\end{equation*}
	where $v$ and $V$ are related by Equation~\eqref{eq:scaling_v}. From this scaling relation we deduce that $\Nrm{v(t,\cdot)}{L^p(\tilde{\Omega}_t)} = (T-t)^{\frac{2\beta}{p} - \alpha} \Nrm{V(t,\cdot)}{L^p(\tilde{\Omega}_t)}$ and if $V$ solves the linearized equation around the state $U^\circ$, then $v$ solves the linearized Euler equation~\eqref{eq:linear_Euler}. Thus, from the definition of the scaled growth bound~\eqref{def:scaled_growth} we have
	\begin{align*}
		\lambda_p(t) = \(T-t\)^{\frac{2\beta}{p} - \alpha} \Lambda_p(t).
	\end{align*}
	Therefore, we can use Proposition~\ref{prop:bound_omega_axi} and Proposition~\ref{prop:bound_beta_axi} with $\omega = \curl u^\circ$ to get
	\begin{equation}\label{eq:bound_omega_scaled}
		\Nrm{\omega(t,\cdot)}{L^\infty} \lesssim \Nrm{\omega^\init}{L^\infty} \(T-t\)^{2\(\frac{2\beta}{p} - \alpha\)} \Lambda_p(T)^2.
	\end{equation}
	Moreover, we can also compute the vorticity from the Formula~\eqref{eq:scaling}. This yields
	\begin{equation*}
		\Nrm{\omega(t,\cdot)}{L^\infty} = \frac{1}{\(T-t\)^{\alpha+\beta}}\,\cC(t),
	\end{equation*}
	with $\cC(t) = \Nrm{\curl U^\circ(t,\cdot)}{L^\infty}$. Thus, we deduce from~\eqref{eq:bound_omega_scaled} that if $u^\circ$ blows-up at $t=T$, then for any $0<T-\eps<t<T$,
	\begin{equation*}
		\Lambda_p(t)^2 \geq \(T-t\)^{\alpha -\beta\(1+\frac{4}{p}\)} \cC_T,
	\end{equation*}
	where $\cC_T = \inf_{[T-\eps,T]}\cC(t )> 0$. Thus, the solution is unstable  as soon as $\beta>0$ and 
	\begin{align*}
		\frac{\alpha}{\beta} < 1+\frac{4}{p}.
	\end{align*}

	\begin{remark}
	 	In the work~\cite{luo_potentially_2014}, the axisymmetric locally self-similar blow-up profile as a slightly more precise shape since different rates are taken for $u_\theta$, and $\omega_\theta$, which yields different scaling for $u_r$ and $u_z$ compared to $u_\theta$. In the scaling of the dominant components of such a solution, we have with the notations of~\cite[Equations~(4.17)]{luo_potentially_2014} $\alpha = -\gamma_u$ and $\beta = \gamma_l$, and the balance of the dominant terms implies $\alpha + \frac{\beta}{2} = 1$ (see~\cite[Equation~(4.20)]{luo_potentially_2014}). Hence, in this case, the solution is unstable if
		\begin{equation*}
			\beta > \frac{1}{\frac{3}{2} + \frac{4}{p}}.
		\end{equation*}
		In particular, this means that at this scale, solutions are unstable in any $L^p$ as soon as $\beta > \frac{2}{3}$. As indicated in the above mentioned paper and proved in~\cite{constantin_geometric_1994}, in the case of a blowing-up solution, $\beta$ is always larger or equal to $1$, so this kind of solutions is always linearly unstable.
	\end{remark}


\bibliography{Euler}

\begin{thebibliography}{10}

\bibitem{beale_remarks_1984}
J.~T. Beale, T.~Kato, and A.~Majda.
\newblock Remarks on the breakdown of smooth solutions for the 3-{D} {Euler}
  equations.
\newblock {\em Communications in Mathematical Physics}, 94(1):61--66, March
  1984.

\bibitem{chae_remarks_2005}
Dongho Chae.
\newblock Remarks on the blow-up criterion of the three-dimensional {Euler}
  equations.
\newblock {\em Nonlinearity}, 18(3):1021--1029, February 2005.

\bibitem{chae_breakdown_1996}
Dongho Chae and Namkwon Kim.
\newblock On the breakdown of axisymmetric smooth solutions for the 3-{D}
  {Euler} equations.
\newblock {\em Communications in Mathematical Physics}, 178(2):391--398, May
  1996.

\bibitem{chen_finite_2019}
Jiajie Chen and Thomas~Y. Hou.
\newblock Finite time blowup of {2D} {Boussinesq} and {3D} {Euler} equations
  with ${C}^{1,\alpha}$ velocity and boundary.
\newblock {\em Arxiv 1910.00173}, November 2019.
\newblock arXiv: 1910.00173.

\bibitem{constantin_geometric_1994}
Peter Constantin.
\newblock Geometric {Statistics} in {Turbulence}.
\newblock {\em SIAM Review}, 36(1):73--98, March 1994.

\bibitem{dyda_muckenhoupt_2019}
Bartłomiej Dyda, Lizaveta Ihnatsyeva, Juha Lehrbäck, Heli Tuominen, and
  Antti~V. Vähäkangas.
\newblock Muckenhoupt {Ap}-properties of {Distance} {Functions} and
  {Applications} to {Hardy}–{Sobolev} -type {Inequalities}.
\newblock {\em Potential Analysis}, 50(1):83--105, January 2019.

\bibitem{elgindi_finite-time_2019}
Tarek~M. Elgindi.
\newblock Finite-{Time} {Singularity} {Formation} for ${C}^{1,\alpha}$
  {Solutions} to the {Incompressible} {Euler} {Equations} on $\mathbb{R}^3$.
\newblock {\em arXiv:1904.04795 [physics]}, April 2019.
\newblock arXiv: 1904.04795.

\bibitem{elgindi_stability_2019}
Tarek~M. Elgindi, Tej-Eddine Ghoul, and Nader Masmoudi.
\newblock On the {Stability} of {Self}-similar {Blow}-up for ${C}^{1,\alpha}$
  {Solutions} to the {Incompressible} {Euler} {Equations} on $\mathbb{R}^3$.
\newblock {\em arXiv:1910.14071 [physics]}, October 2019.
\newblock arXiv: 1910.14071.

\bibitem{friedlander_nonlinear_1997}
Susan Friedlander, Walter~A. Strauss, and Misha Vishik.
\newblock Nonlinear instability in an ideal fluid.
\newblock {\em Annales de l'Institut Henri Poincaré. Analyse Non Linéaire},
  14(2):187--209, 1997.

\bibitem{friedlander_dynamo_1991}
Susan Friedlander and Misha Vishik.
\newblock Dynamo theory, vorticity generation, and exponential stretching.
\newblock {\em Chaos. An Interdisciplinary Journal of Nonlinear Science},
  1(2):198--205, 1991.

\bibitem{friedlander_instability_1991}
Susan Friedlander and Misha Vishik.
\newblock Instability criteria for the flow of an inviscid incompressible
  fluid.
\newblock {\em Physical Review Letters}, 66(17):2204--2206, 1991.

\bibitem{hou_dynamic_2006}
Thomas~Y. Hou and Ruo Li.
\newblock Dynamic depletion of vortex stretching and non-blowup of the 3-{D}
  incompressible {Euler} equations.
\newblock {\em Journal of Nonlinear Science}, 16(6):639--664, 2006.

\bibitem{inoue_existence_1979}
Atsushi Inoue and Tetsuro Miyakawa.
\newblock On the existence of solutions for linearized {Euler}'s equation.
\newblock {\em Japan Academy. Proceedings. Series A. Mathematical Sciences},
  55(8):282--285, 1979.

\bibitem{kerr_bounds_2013}
Robert~M. Kerr.
\newblock Bounds for {Euler} from vorticity moments and line divergence.
\newblock {\em Journal of Fluid Mechanics}, 729:R2, 13, 2013.

\bibitem{luo_potentially_2014}
Guo Luo and Thomas~Y. Hou.
\newblock Potentially singular solutions of the {3D} axisymmetric {Euler}
  equations.
\newblock {\em Proceedings of the National Academy of Sciences},
  111(36):12968--12973, September 2014.

\bibitem{mazya_sobolev_2011}
Vladimir Maz'ya.
\newblock {\em Sobolev {Spaces}}, volume 342 of {\em Grundlehren der
  mathematischen {Wissenschaften}}.
\newblock Springer Berlin Heidelberg, Berlin, Heidelberg, 2011.

\bibitem{vasseur_blow-up_2019}
Alexis~F. Vasseur and Misha Vishik.
\newblock Blow-up solutions to 3{D} {E}uler are hydrodynamically unstable.
\newblock {\em Comm. Math. Phys.}, 378(1):557--568, 2020.

\bibitem{vishik_spectrum_1996}
Misha Vishik.
\newblock Spectrum of small oscillations of an ideal fluid and {Lyapunov}
  exponents.
\newblock {\em Journal de mathématiques pures et appliquées}, 75(6):531--557,
  1996.

\bibitem{yang_weighted_2020}
Sibei Yang, Der-Chen Chang, Dachun Yang, and Wen Yuan.
\newblock Weighted gradient estimates for elliptic problems with {Neumann}
  boundary conditions in {Lipschitz} and (semi-)convex domains.
\newblock {\em Journal of Differential Equations}, 268(6):2510--2550, March
  2020.

\end{thebibliography}
	
\end{document}